\documentclass[11pt]{article}

\usepackage{amssymb,amsmath,amsfonts} 
\usepackage{tikz}

\usetikzlibrary{calc,shapes.geometric} % required for the polygon shape
\usetikzlibrary{trees,positioning,arrows,shapes,snakes}

\pgfkeys{/tikz/.cd,
     num vertex/.initial=4,
     num vertex/.get=\vertices,
     num vertex/.store in=\vertices,
     circle radius/.initial=3,
     circle radius/.get=\circleradius,
     circle radius/.store in=\circleradius,
     shift angle/.initial=0,
     shift angle/.get=\shiftangle,
     shift angle/.store in=\shiftangle, 
     at pos/.initial={(0,0)},
     at pos/.get=\position,
     at pos/.store in=\position,
     vertex radius/.initial=1.5pt,
     vertex radius/.get=\vertexradius,
     vertex radius/.store in=\vertexradius,
}

\makeatletter
\def\drawvertices{\tikz@path@overlay{node}}
\makeatother   

\pgfkeys{/tikz/circumference with labels in order/.code={
        \pgfmathsetmacro\halfcircleradius{\circleradius/2}
        \draw \position circle (\halfcircleradius cm) node[regular polygon, regular polygon sides=\vertices, minimum size=\circleradius cm, draw=none, name={vertex set}] {};
        \foreach \circlecolor [count=\x] in {#1}{
            \node[draw,circle, inner sep=\vertexradius,black, fill=\circlecolor] at (vertex set.corner \x) {};
            \pgfmathparse{\shiftangle-360*(\x-1)/ \vertices}
            \node at ($(vertex set)+(\pgfmathresult:\halfcircleradius)$){};
        }
    }
}

\newtheorem{theorem}{Theorem}
\newtheorem{lemma}{Lemma}
\newtheorem{corollary}{Corollary} [theorem]

\newtheorem{property}{Property}

\newcommand{\qed}{\ifhmode\unskip\nobreak\fi\ifmmode\ifinner
\else\hskip5 pt\fi\fi \hbox{\hskip5 pt
\vrule width4 pt  height6 pt  depth1.5 pt \hskip 1pt }}

\newenvironment{myalgorithm}
{\par\indent\vbox\bgroup\small\kern -1\baselineskip
\begin{tabbing}
\ \ \ \ \ \=\ \ \ \ \=\ \ \ \ \=\ \ \ \ \=
\ \ \ \ \=\ \ \ \ \=\ \ \ \ \=\ \ \ \ \=\ \ \ \ \=
\ \ \ \ \=\ \ \ \ \=\ \ \ \ \=\ \ \ \ \=\\}
{\end{tabbing}
\kern -1\baselineskip\egroup}
\newcommand{\Algoritmo}{{\bfseries Algorithm }}
\newcommand{\Procedimento}{{\bfseries procedure }}

\newcommand{\Se}{{\bfseries if }}

\newcommand{\entao}{{\bfseries then }}
\newcommand{\senao}{{\bfseries else }}
\newcommand{\Para}{{\bfseries for }}

\newcommand{\faca}{{\bfseries do}}
\newcommand{\Inicio}{{\bfseries begin}}
\newcommand{\Fim}{{\bfseries end}}

\newcommand{\shrink}{\kern -0.5\baselineskip}

\title{Linear time determination of the scattering number for  strictly chordal graphs}
\author{Lilian Markenzon \footnote{Partially supported by grant 304706/2017-5, CNPq, Brazil.}\\ 
       NCE -  Universidade Federal do Rio de Janeiro\\
                     markenzon@nce.ufrj.br
\and
            Christina F. E. M. Waga\\
            IME - Universidade do Estado do Rio de Janeiro   \\
            waga@ime.uerj.br           
}
\date{\ }

\begin{document}
\maketitle

\begin{abstract}

The scattering number of a graph $G$ was defined by  Jung in 1978 as  
$sc(G) = max \{  \omega(G - S) - |S|, S \subseteq V,  \omega(G - S) \neq1\}$  
 where  $\omega(G - S) $ is  the number of connected components  of the graph $G-S$. 
 It is a measure of vulnerability of a graph and it has a direct relationship with the toughness of a graph. 
 Strictly chordal graphs, also known as  block duplicate graphs, are a subclass of chordal graphs
  that includes  block and 3-leaf power graphs. 
 In this paper we present a  linear time solution for the determination of 
the scattering number and    scattering set of   strictly chordal graphs. 
We show that, although the knowledge of the toughness of the class is helpful, it is not sufficient to provide an immediate result for determining the scattering number.

 \end{abstract}

Keywords:
strictly chordal graph;  scattering number;   toughness; minimal vertex separator.

%********************************************************************************
%   Introduction                                *
%********************************************************************************
\section{Introduction}\label{section:introd}

Vulnerability in graphs  is mainly related to the study of a graph when some of its elements are removed. 
There are many well known measures of  vulnerability based on subsets of vertices. 
Some are more usual as connectivity,  domination number, domatic number and  independence number; another ones are more recent as toughness \cite{Ch73}, binding number \cite{W73}, scattering number  \cite{Ju78},   
integrity \cite{BES87}, tenacity \cite{CMS95} and rupture degree \cite{LZL05}.  
Relationships between these invariants were studied  in  \cite{BES87} and  \cite{ZP04}. 

Toughness and  scattering number remind vertex connectivity since both consider the cardinality of a separator
 but they also take into account the  number of 
 remaining connected components after its removal.  
The toughness of a graph  was 
introduced by  Chv\' atal in 1973 \cite{Ch73}.
A graph $G=(V,E)$ is $t$-tough if
 $|S| \geq  t \,\omega(G-S)$ for every subset $S\subseteq V$  with  $\omega(G-S) > 1$. 
The \textit{toughness} of $G$, denoted $\tau(G)$, is the maximum value of $t$ for which $G$ is $t$-tough, taking $\tau(K_n) = \infty$,  $n \geq  1$. 
Therefore, if  $G$ is not complete, $\tau(G)= min\bigg\{ \frac{|S|}{\omega(G-S)}  \bigg\}$
where the minimum is taken over all separators  $S$  of vertices in $G$   \cite{B06}.
Subset $S$ for which  this minimum is attained is called a \emph{tough set}.
 The \emph{scattering number} of  a graph $G$ was defined by Jung  in 1978 \cite{Ju78}   
 as $sc(G) = max \{  \omega(G - S) - |S|, S \subseteq V \textrm{ and } \omega(G - S)\neq 1\}$.
 Subset $S$ for which  this maximum is attained is called a \emph{scattering set}. 
In both definitions, $\omega(G - S) $ is  the number of connected components  of the graph $G-S$. 
 Kratsch {\em et al.}  \cite{KKM94} related toughness and scattering number as follows. 

 \begin{lemma} {\rm \cite{KKM94}}  \label{lem:touscat} 
 For every graph $G$ holds $\tau(G) \geq 1$ if and only if $sc(G) \leq 0$. 
 \end{lemma}
Furthermore,  they presented an important result which  provides an algorithmic approach 
to  computing the scattering number. 

\begin{theorem} {\rm \cite{KKM94}}  \label{theo:scformula}
Let $G=(V,E)$ be a graph which is not complete. Then 
$$sc(G) = max_{_{S}} \Bigg\{  \sum_{i=1}^{k} max \{ sc ( G\left[ C_i \right] )  ,1 \} - |S|    \Bigg\}$$
where the maximum is taken over all minimal separators $S$ of the graph $G$ and $C_1, \dots, C_k$
 are the connected components of $G\left[ V\setminus S \right]$.
\end{theorem}

For general graphs,
the corresponding computational problems to determine these invariants   are in NP \cite{BHS90, ZLH02}.
Few classes of graphs present polynomial results when dealing with the determination of the  toughness:  
interval graphs and trapezoid graphs \cite{KKM97}, cocomparability graph  with  $\tau(G) \geq 1$ \cite{DKS97} and  claw-free graphs,
 split graphs and $2K_2$-free graphs \cite{B15}.   Markenzon and Waga \cite{MW19}  presented a linear time determination of the  toughness of strictly chordal graphs. 
The scattering number of  interval graphs  \cite{BFGKPP15}, grid graphs and cartesian product of two complete graphs \cite{ZLH02} and gear graphs   \cite{AK11}  can be solved in linear time; however,  trapezoid graphs maintain the polynomial result to the scattering number  \cite{KKM97}. 
Observe that this invariant has some variations: edge scattering number \cite{A11}  opts for  $S$ as a  subset of edges of the graph;   isolated scattering number  \cite{Wetall11,LYS17} focus on  components that are isolated vertices of $G-S$ and 
weighted scattering number \cite{LZB19} takes into account the importance of the vertices in the graph establishing weight for them. 

The  {\em block duplicate graphs}, a subclass of ptolemaic graphs,  were introduced by  Golumbic and Peled \cite{GP02}. 
The class was also  defined as {\em strictly chordal graphs} by Kennedy  in \cite{K05} based on hypergraph properties and
it was proved to be gem-free and  dart-free \cite{GP02,K05}.
Brandst\"adt and  Wagner \cite{BW10}  showed  that the class is  the same as the  $(4,6)$-leaf power graphs.  
Strictly chordal graphs  can also be characterized in terms of the structure of their separators \cite{MW15}. 
Some known  subclasses of strictly chordal graphs are:  block graphs \cite{H63}, AC-graphs \cite{B93},
  3-leaf power graphs \cite{DGHN04,BL10}, strictly interval graphs \cite{MW16} and generalized core-satellite graphs \cite{EB17}. 

In his seminal paper, Jung stated that ``the scattering number is in a certain sense the {\em additive dual} for
the concept of toughness''.
Hence, as the toughness of a strictly chordal graph is already known, 
we study  the determination of 
its scattering number exploring this relationship. 
However, although the determination of the toughness is quite simple,
our proposed  task was not straightforward.
Firstly we show that, for $\tau(G) \geq 1$, the determination of the scattering number is quite
similar to the determination of the toughness;
the scattering set is also composed by a sole minimal vertex separator of $G$ 
but a tough set is seldom a scattering set.
For $\tau(G) < 1$, we need to establish a further partition of the set of graphs. 
After setting apart the graphs for which the scattering number is equal to one (we call them \emph{type A} graphs)
  for the remaining graphs (\emph{type B} graphs) a more algorithmic approach  is required.
For them, a scattering set must be build  in order to determine their scattering number;
this set can be  composed by one or more minimal vertex separators of $G$.
In all cases our solution has  linear time complexity.

%*******************************************************************************
%   Background                                *
%*******************************************************************************

\section{Background}\textbf{}

Basic concepts about chordal graphs (graphs possessing no chordless cycles)
   are assumed to be known and 
can be found in  Blair and Peyton \cite{BP93} and Golumbic \cite{Go04}.  
In this section, the most pertinent concepts are reviewed.

Let $G=(V,E)$  be a connected graph, 
where $|E|=m$  and
 $|V| = n$. 
The {\em neighborhood\/} of a vertex $v \in V$ is denoted by
$N(v) = \{ w \in V; \{v,w\} \in E\}$ and its {\em closed neighborhood} by  $N[v] = N(v)\cup \{v\} $.
Two vertices $u$ and $v$   are  {\em true twins} in $G$  if $N[u] = N [v]$.
A vertex $v$ is said to be {\em
simplicial\/} in $G$ when $N(v)$ is a {\em clique\/} in $G$.
For any $H \subseteq V$, 
  the subgraph of $G$ induced by $H$ is denoted $G[H]$. 

Let $G = (V, E)$ be a chordal graph and $u,v  \in V$. A subset $S \subset V$ is
a {\em separator} of $G$ if at least two vertices in the same connected
component of $G$ are in two distinct connected components of
$G[V\setminus S]$; 
$S$ is a {\em minimal separator} of $G$ if $S$
is a separator and no proper subset of $S$ separates the graph.
A subset $S \subset V$  is a {\em vertex separator}  for non-adjacent
vertices $u$  and $v$  (a {\em $uv$-separator}) if the removal of $S$ from
the graph separates $u$ and $v$  into distinct connected components.
If no proper subset of $S$  is a $uv$-separator then $S$  is a {\em
minimal $uv$-separator}. 
If $S$ is a minimal $uv$-separator for some pair of vertices,   it is called a
{\em minimal vertex separator} (\textit{mvs}). 
A minimal separator is always a minimal vertex separator
but the converse is not true.

A {\em clique-tree} of $G$ is defined as a tree  $T=(\mathbb{Q}, E_T)$,  where $\mathbb{Q}$ 
is the set of maximal cliques of $G$ 
and for every  two distinct  maximal cliques $Q, Q^\prime \in \mathbb{Q}$
each clique in the path from $Q$ to $Q^\prime$ in $T$ contains $Q\cap Q^\prime$. 
Observe that a set $S\subset V$ is a minimal vertex separator of $G$ if
and only if $S= Q\cap Q' $ for some edge $\{Q, Q'\}\in E_T$. 
Moreover, the multiset  ${\mathbb M}$ of
the minimal vertex separators of $G$ is the same for every
clique-tree of $G$.
The {\em multiplicity} of the minimal vertex separator $S$, denoted by
$\mu(S)$, is the number of times that $S$ appears in  ${\mathbb M}$. 
The set of minimal vertex separators  of $G$ is denoted by $\mathbb{S}$.
The determination of the minimal vertex separators and their multiplicities 
can be performed in linear time  \cite{MP10}. 

It is  important to mention two types  of cliques in a chordal graph $G$. 
A   {\em simplicial clique} is a maximal clique containing at least one simplicial vertex.
A simplicial clique $Q$ is called a {\em boundary clique} if there exists a maximal clique $Q^\prime$
such that  $Q \cap Q^\prime $ is the set of non-simplicial vertices of $Q$. 

A {\em strictly chordal graph} is a graph obtained by adding zero or more true twins 
to each vertex of a block graph $G$  \cite{GP02}.
The class was 
proved to be (gem,dart)-free \cite{GP02, K05}.
Strictly chordal graphs  can also be characterized in terms of the structure of their minimal vertex separators as proved in Theorem \ref{theo:caract2}. 

\begin{theorem}\label{theo:caract2} {\rm \cite{MW19} }
Let  $G=(V,E)$ be a chordal graph and $\mathbb S$ be the set of minimal vertex separators of $G$.
$G$ is a strictly chordal graph if and only if for any distinct $S, S^{\prime} \in {\mathbb S}$, $S \cap S^{\prime}= \emptyset$.
\end{theorem}

The \emph{clique-bipartite graph}
of $G$  is the bipartite graph $CB(G) = (\mathbb{S} \cup \mathbb{Q}, F)$ in which there is an edge joining a maximal clique $Q\in \mathbb{Q}$ and a minimal separator  $S\in \mathbb{S}$ when  $S\subset Q$  \cite{M18}.
This structure is a generalization of the block-cut vertex graph defined by Harary \cite{H69}. 
Theorem \ref{theo:BC}  is an immediate consequence of  results about  clique-bipartite graphs presented in  \cite{M18}.

 \begin{theorem}\label{theo:BC} 
 Let $G$ be a  strictly chordal graph. Then  the clique-bipartite graph  $CB(G)$ is  a tree. 
\end{theorem}
 
The determination of the toughness of strictly chordal graphs, as seen in Theorem \ref{theo:touBD}, 
 can be performed in linear time complexity.

\begin{theorem} \label{theo:touBD} {\rm \cite{MW19} }
Let $G$ be a non-complete  strictly chordal graph and $\mathbb S$ be the set of minimal vertex separators of $G$.
Then   $\tau(G)=min_{S \in \mathbb S} \bigg\{  \frac{|S|}{\mu(S) +1}\bigg\}$.
\end{theorem}

%**************************************************
%Scattering
%*************************************************

\section{Scattering number of strictly chordal graphs}

As already seen in Section \ref{section:introd},  the scattering number of $G$  is 
$$sc(G) = max \{  \omega(G - S) - |S|; S \subseteq V,  \omega(G - S) \neq1\}$$
where  $\omega(G - S) $ is  the number of connected components  of the graph $G-S$ and a  subset   $S\subset V$  for which this maximum is attained is 
a \emph{scattering set} of $G$. 

\smallskip 

Some basic properties of the separators of  strictly chordal graphs will be fundamental for 
the determination of their scattering numbers. 
They can be stated:

 \begin{property}  \label{prop:sc}  Let $G=(V,E)$ be  a strictly chordal graph and $\mathbb S$ its set of minimal vertex separators.
 \begin{enumerate}
 \item[a)] for any distinct $S, S^{\prime} \in {\mathbb S}$, $S \cap S^{\prime}= \emptyset$.
 
 \item[b)]  $S \in {\mathbb S}$  is a minimal separator of $G$.
 
 \item[c)] boundary cliques of $G$ contain only one mvs. 
 
  \item[d)] $\omega(G-S)=\mu(S)+1$, for every  $S \in {\mathbb S}$.
  
  \item[e)]  every separator $\mathbf{S}$ of $G$  is the  union of  pairwise disjoint minimal vertex separators  of $G$. 
  \end{enumerate}
 \end{property}
 
  The maximal cliques that contain $S \in {\mathbb S}$ are called {\em adjacent cliques of $S$}; the cardinality of this set is $\mu(S)+1$.
 The  set of  boundary cliques that contains the   {\em mvs}  $S\in {\mathbb S}$
 is denoted by  $B(S)$. 

In Theorem \ref{theo:specialS} it is proved  the first  result relating the minimal vertex separators  and the boundary cliques of $G$.
It will be used for the development of an efficient algorithm to determine a scattering set.

 \begin{theorem}  \label{theo:specialS} 
 Let $G$ be a   strictly chordal graph with $|\mathbb{S}| >1$. 
 Then  there is  a mvs     $S\in {\mathbb S}$ of $G$ such that $|B(S)|=\mu(S)$.
 \end{theorem}
 
\begin{proof}
As $G$ is a   strictly chordal graph, every boundary clique in $G$ has only one {\em mvs}. 
Let $CB(G)$ be the clique-bipartite graph of $G$. 
The  boundary cliques of $G$ are the leaves of $CB(G)$ and for every \emph{mvs} $S\in \mathbb{S}$, 
the degree of the vertex that represents $S$ in $CB(G)$ is $\mu (S)+1$. 

Suppose that there is not a \emph{mvs} $S$ of $G$ such that $|B(S)|=\mu(S)$, i.e., 
every \emph{mvs}  of $G$ is a subset of  at least two  maximal cliques  that are  not  boundary cliques.  
However, the maximal cliques that are not boundary cliques have at least two minimal vertex separators. 
Since $G$  has a finite number of maximal  cliques, $CB(G)$ must  contain  a cycle. 
 Contradiction, by Theorem \ref{theo:BC}. 
Then  there is a {\em mvs}   $S$ of $G$ such that  $|B(S)|=\mu(S)$.\qed\end{proof}

\medskip

Any  \emph{mvs} $S\in {\mathbb S}$ described in Theorem \ref{theo:specialS}  is called a  \emph{border minimal vertex separator} of $G$.

From now on, the determination  of the scattering number and the scattering set  of strictly chordal graphs is addressed.
If  the graph  has only one minimal vertex separator, the result is immediate. 

 \begin{theorem}  Let $G$ be a non-complete strictly chordal graph.
  If $ |{\mathbb S}| =1$ then    $sc(G) = \mu(S)+1 -|S|$ with  $S\in \mathbb S$.  
\end{theorem}

In the remaining of this section  we  will consider graphs  with at least two minimal vertex separators and  our approach will be the analysis of 
the graph according to its  toughness. 

%%%%%%%%
%%%%%%   tau > 1 
%%%%%%%%

\subsection{graphs with {\boldmath{$\tau(G) \geq 1$}}}\label{subsection-greather}

\begin{theorem}\label{theo:tau>=1} 
Let $G$ be  a strictly chordal graph with  $|\mathbb{S}| >1$ and   $\tau(G) \geq 1$. 
 Then 
 $sc(G) = max_{_{S\in \mathbb S}}  \{ \mu(S) +1 -|S| \}$. 
 \end{theorem}
\begin{proof} As   $\tau(G) \geq 1$, for every \emph{mvs} $S' \in \mathbb S$, $|S'| \geq  \mu(S') +1$ and 
$sc (G) \leq 0$ (Lemma \ref{lem:touscat}).

Consider a \textit{mvs}  $S\in \mathbb S$ such that $\alpha = \mu(S) +1-|S|$ is the highest  possible; $\alpha \leq 0$. 
 Since the  scattering number is defined as the maximum, let us analyse a separator of $G$ that is  the union of $S$ and   other \emph{mvs} $S'\in \mathbb S$, $S'\neq S$.
For every \emph{mvs} $S' \in \mathbb S$, $ \mu(S') +1-|S'| \leq \alpha \leq  0$. 
So, $\omega(G- \{ S\cup S^{\prime}\})-|S \cup S'|<\omega(G- S)- |S|=\alpha$,  i.e., 
 any union  presents  an   worse result than the result obtained by $S$. 
Then,    $S$ is a  scattering set and   $sc(G) = max_{_{S \in \mathbb S }}  \{ \mu(S) +1 -|S| \}$.   \qed\end{proof}\\

Observe  that,  for a graph $G$, $ sc(G) = 0 $     if and only if  $\tau(G) = 1 $.
 
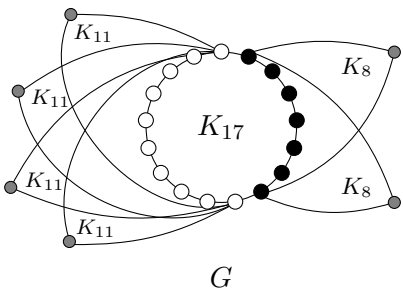
\begin{figure}[!h] 
\begin{center}
\begin{tikzpicture}
\coordinate (Oa) at (0.5,1.5);
\coordinate (Ob) at (-0.2,0.45);
\coordinate (Oc) at (-0.3,-0.8);
\coordinate (Od) at (0.5,-1.56);
\coordinate (Oe) at (4.8,1);
\coordinate (Of) at (4.8,-1);
\coordinate (A) at (2.5,1);
\coordinate (B) at (2.7,-1);
\coordinate (A1) at (2.9,1);
\coordinate (B1) at (3,-0.9);

\draw[] (Oa) to [bend left=15] (A);
\draw[] (Oa) to [bend right=68] (B);
\draw[] (Ob) to [bend left=25] (A);
\draw[] (Ob) to [bend right=55] (B);
\draw[] (Oc) to [bend left=30] (A);
\draw[] (Oc) to [bend right=20] (B);
\draw[] (Od) to [bend left=58] (A);
\draw[] (Od) to [bend right=15] (B);
\draw[] (Oe) to [bend right=15] (A1);
\draw[] (Oe) to [bend left=30] (B1);
\draw[] (Of) to [bend right=25] (A1);
\draw[] (Of) to [bend left=20] (B1);

\node at (0.8,1.25) [draw=none,fill=none] {\scriptsize{$K_{11}$}}; 
\node at (0.23,0.38) [draw=none,fill=none] {\scriptsize{$K_{11}$}}; 
\node at (0.15,-0.72) [draw=none,fill=none] {\scriptsize{$K_{11}$}}; 
\node at (0.83,-1.35) [draw=none,fill=none] {\scriptsize{$K_{11}$}}; 
\node at (4.3,0.8) [draw=none,fill=none] {\footnotesize{$K_{8}$}}; 
\node at (4.3,-0.8) [draw=none,fill=none] {\footnotesize{$K_{8}$}}; 
\node at (2.5,0) [draw=none,fill=none] {$K_{17}$}; 

\draw[black, fill=gray] (0.5,1.5) circle (0.8mm);
\draw[black, fill=gray] (-0.2,0.48) circle (0.8mm);
\draw[black, fill=gray] (-0.3,-0.8) circle (0.8mm);
\draw[black, fill=gray] (0.48,-1.52) circle (0.8mm);
\draw[black, fill=gray] (4.8,1) circle (0.8mm);
\draw[black, fill=gray] (4.8,-1) circle (0.8mm);
\drawvertices[num vertex=17, 
    circle radius=2,
    vertex radius=2pt,
    shift angle=90,
    at pos={(2.5,0)},
    circumference with labels in order={
    white,white,white,white,white,white,white,white,white,white,black,black,black,black,black,black,black}] {};     
  
    \node at (2.5,-2) [draw=none,fill=none] {$G$};  
\end{tikzpicture}  
\end{center}
\vskip-0.5cm
\caption{$\tau(G)\geq 1$}
\label{fig:tou>1}
\end{figure}

As it was seen in Theorems \ref{theo:touBD} and \ref{theo:tau>=1}, the determination of the toughness and the determination of the scattering number of a graph $G$
with  $\tau(G) \geq 1$ are quite similar.  
Observe that, for $\tau(G)>1$, a tough set is not always a scattering set;
an example is shown in Figure \ref{fig:tou>1}.
Graph $G$ has  $\tau(G)=2$,  its tough set is composed  by the white vertices, 
  $sc(G)=-4$ and the  scattering set is composed  by the black vertices.  

%%%%%%%%
%%%%%%   tau < 1 
%%%%%%%%

\subsection{graphs with  {\boldmath{$\tau(G) < 1$}}}\label{subsection:tau<1}

In Subsection \ref{subsection-greather} it was seen that  a graph $G$ with  $\tau(G) \geq 1$ have
all their minimal vertex separators such that  $|S| > \mu(S)$.
For graphs with $\tau(G) < 1$, the cardinality of 
the minimal vertex separators can be greater, equal or lower than their multiplicity;
however  it is mandatory to have at least one \emph{mvs} $S$ with $S \leq \mu(S)$.
In order to analyse this case we need to establish
a further partition of the set of graphs.
Let $G$ be a strictly chordal graph with $\tau(G)<1$  and $\mathbb S$ its set
of minimal vertex separators: 
graph $G$ is called a {\em type A graph} if $|S| \geq \mu(S)$, for every \emph{mvs}  $S \in \mathbb S$; 
otherwise it is called a {\em type B graph},  i.e., there exists at least one  \emph{mvs}   $S \in \mathbb{S}$
 such that $|S| < \mu(S)$. 

The determination of  the scattering number of type A graphs is quite simple as we can see in  the following theorem.

 \begin{theorem} \label{theo:sctype A}
 Let $G$ be a type A  graph.
 Then $sc(G) = 1$. 
 \end{theorem}
    \begin{proof}  
Graph  $G$ is  a type A  graph,   $\tau(G) < 1 \therefore sc(G) >0 \therefore sc(G) \geq 1$ (Lemma \ref{lem:touscat}) and  for every  $S'\in \mathbb{S}$,  $|S'| \geq \mu(S')$. 
 As $\tau(G) < 1$, there is at least a  \emph{mvs} $S$ such that $ |S| =  \mu (S)$ and $\tau(G)= \frac{|S|}{|S|+1}$. So, 
 $\omega(G-S)-|S|=\mu(S) +1 -|S|=1$. 
  In fact, for every $S'\in \mathbb{S}$ such that  $|S'| = \mu(S')$, $\mu(S') +1 -|S'|=1$. 

Consider the separator  $\mathbf{S}=S' \cup S''$ with   $S', S'' \in \mathbb S$, $S'\neq S''$. 
 By Theorem \ref{theo:caract2},  $|\mathbf{S}|=|S'|+|S''|$. 
Firstly  let us determine the number of components of $G-\mathbf{S}$. 
 If $\mathbf{S}$ is a maximal clique of $G$, 
 $\omega(G-\mathbf{S})= \omega(G-S')+ \omega(G-S'')-2=\mu(S')+\mu(S'')$;  otherwise, $\omega(G-\mathbf{S})= \omega(G-S')+ \omega(G-S'')-1=\mu(S')+\mu(S'')-1$. 
 If  $|S'| >   \mu(S')$ or  $|S''| >   \mu(S'')$, it is immediate that  $\omega(G-\mathbf{S})-|\mathbf{S}| \leq 0$. 
If   $|S'|=  \mu(S')$ and  $|S''| =  \mu(S'')$,  
$\omega(G- \{ S\cup S^{\prime}\})-|\mathbf{S}|\leq 1$. 
Consider $\mathcal{S}$ is  a separator of $G$ with at least three minimal vertex separators of  $\mathbb{S}$  and such that for every $S'\subset   \mathcal{S}$, $ |S'|=\mu(S')$. 
 By the same reasoning,   
$\omega(G-\mathcal{S})-|\mathcal{S}  | \leq 1$.  
Then, $sc(G) = 1$.  \qed\end{proof}

\begin{corollary}\label{corol:tau=1A}
 Let $G$ be a type A  graph. Then every  mvs $S \in  \mathbb{S}$  such that  $|S| = \mu(S)$  is a scattering set.
  \end{corollary}

\begin{figure}[!h] 
\begin{center}
\begin{tikzpicture}
 \tikzstyle{vb}=[
draw,
fill=black,
circle,    inner sep=0pt, minimum width=5pt,
align=center]

\coordinate (Oa) at (2.17,1);
%\coordinate (Ob) at (1.725,0.55);
\coordinate (Oc) at (1.5,0);
\coordinate (Od) at (1.725,-0.55);
\coordinate (Oe) at (2.17,-1);
\coordinate (Og) at (3.4,-0.55);
\coordinate (Oh) at (3.5,0);
\coordinate (Oi) at (3.4,0.55);
\coordinate (Oj) at (2.83,1);
              
  \node [vb]  (a) at (2,1.75) {};   \node [vb]  (j) at (3,1.75) {}; 
  \node [vb]  (c) at (0.75,0) {};    \node [vb]  (h) at (4.25,0.5) {}; \node [vb]  (g) at (4.25,-0.5) {}; 
  \node [vb]  (d) at (1.5,-1.5) {}; 

\node at (2.5,0) [draw=none,fill=none] {$K_{10}$}; 
\node at (2.25,0.7) [draw=none,fill=none] {\footnotesize{$a$}}; 
\node at (2.8,0.7) [draw=none,fill=none] {\footnotesize{$j$}}; 
\node at (1.9,0.5) [draw=none,fill=none] {\footnotesize{$b$}}; 
\node at (1.75,0) [draw=none,fill=none] {\footnotesize{$c$}}; 
\node at (1.94,-0.45) [draw=none,fill=none] {\footnotesize{$d$}}; 
\node at (2.25,-0.7) [draw=none,fill=none] {\footnotesize{$e$}}; 
\node at (2.8,-0.7) [draw=none,fill=none] {\footnotesize{$f$}}; 
\node at (3.12,-0.45) [draw=none,fill=none] {\footnotesize{$g$}}; 
\node at (3.25,0) [draw=none,fill=none] {\footnotesize{$h$}}; 
\node at (3.12,0.5) [draw=none,fill=none] {\footnotesize{$i$}}; 
\drawvertices[num vertex=10, 
    circle radius=2,
    vertex radius=1.7pt,
    shift angle=90,
    at pos={(2.5,0)},
    circumference with labels in order={
    black,black,black,black,black,black,black,black,black,black}] {};
    
   \foreach \from/\to in {a/Oa,a/Oj,j/Oa,j/Oj,c/Oc,d/Od,d/Oe,h/Oh,h/Oi,h/Og,g/Oh,g/Oi,g/Og}  
    \draw (\from) -- (\to);  
    \node at (2.5,-2.1) [draw=none,fill=none] {$G_1$};  
    
                               %  primeira linha
  \node [vb]  (a) at (9,0) {}; 
    %  L1
 \node [vb] (b) at (8.5,0.5)  {}; 
 \node [vb] (c) at (9.5,0.5) {};    
   %  L-1
   \node [vb] (d) at (9.5,-0.5) {};
   \node [vb] (e) at (8.5,-0.5) {}; 
    %L1.1
  \node [vb] (f) at (7.7,0.5)  {};   
 \node [vb]  (g) at (8.5,1.2) {}; 
 \node [vb] (h) at (9.5,1.2) {}; 
   \node [vb] (i) at (10.3,0.5) {}; 
    %L-1.1
 \node [vb] (j) at (10.3,-0.5)  {}; 
\node [vb] (k) at (9.5,-1.2)  {}; 
 \node [vb]  (l) at (8.5,-1.2) {}; 
  \node [vb] (m) at (7.7,-0.5) {}; 

    \foreach \from/\to in {a/b,a/c,a/d,a/e,b/f,b/g,c/h,c/i,d/j,d/k,e/m,e/l}  
    \draw (\from) -- (\to);
\node at (8.7,0) [draw=none,fill=none] {\footnotesize{$k$}}; 
\node at (8.35,0.7) [draw=none,fill=none] {\footnotesize{$\ell$}}; 
\node at (9.75,0.7) [draw=none,fill=none] {\footnotesize{$m$}}; 
\node at (9.75,-0.7) [draw=none,fill=none] {\footnotesize{$n$}}; 
\node at (8.35,-0.7) [draw=none,fill=none] {\footnotesize{$o$}}; 
\node at (9,-2.1) [draw=none,fill=none] {$G_2$};  
\end{tikzpicture}  
\end{center}
\vskip-0.5cm
 \caption{$\tau(G_1), \tau(G_2)<1$}
\label{fig:t<1}
\end{figure}
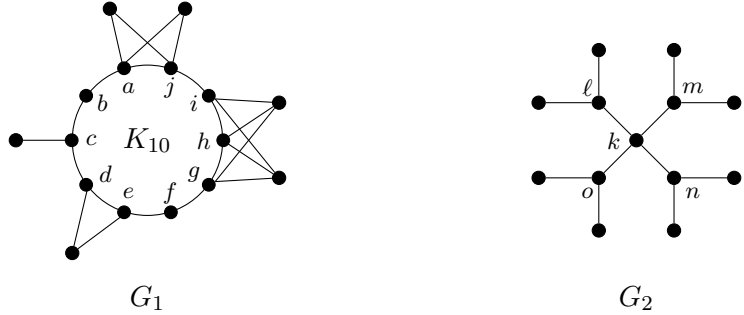

In Figure  \ref{fig:t<1},  graphs $G_1$ and $G_2$ have    toughness  less than $1$. 
Graph $G_1$ is  a  type A graph  with  $\tau(G_1) =.5$, $sc(G_1)=1$  
and three scattering sets $\{c\}$, $\{a,j\}$ and $\{a,c,j\}$. 
Graph $G_2$ is a type B graph with $\tau(G_2) =.25$, the scattering set is $\{\ell,m,n,o \}$ and $sc(G_2)=5$. 

In order to determine the scattering number  of a type B graph, 
Theorems \ref{theo:notborder},   \ref{theo:notborder2} and  \ref{theo:notborder3}  contain a detailed study on border minimal vertex separators.   
Theorem \ref{theo:notborder} provides a criterion to decide whether or not a \emph{mvs} should be considered as part of a scattering set of  strictly chordal graphs in general. Theorems  \ref{theo:notborder2}  and  \ref{theo:notborder3} present results to type B graphs.  

\begin{theorem} \label{theo:notborder}
Let $G$ be  a strictly chordal graph  with $|\mathbb{S}| >1$, $SC$ be a scattering set    
and   $S\in {\mathbb S}$ be a border mvs. 
If   $|S| < | B (S) |$ then   $SC\cup S $ is a scattering set of $G$. 
 \end{theorem}
  
  \begin{proof}  
If $S\subseteq SC$ then  it is immediate that $SC\cup S$  is a scattering set.

Let $S \not\subseteq SC$.
Since  $SC$ is a scattering set of $G$, $\omega(G-SC)-|SC|=s$ is  maximum.
As $S$ is a border \emph{mvs}, $| B (S) |= \mu(S)$. 
Since  $|S| < | B (S) |= \mu(S)$  then  $\tau(G)<1$ and $sc(G)=s>0$.

Let   $G-SC$  be a  graph  with  $C_1,C_2,\dots,C_{\omega(G-SC)}$ connected components. 
Two cases  must be analysed. 

\begin{itemize}
\item  $B(S)=C_i, 1\leq i \leq \omega(G-SC)$.  
Consider the separator $SC\cup S $.
After the removal of this separator, the component $C_i$ does not exist anymore
and there are $|B(S)|$ new components.
The scattering number would become: 
$s'=\omega(G- \{SC \cup S\})-|SC \cup S|= s+|B(S)|-|S|-1. $

By hypothesis, $|S| < |B(S)|$; so, $|B(S)|-|S|-1\geq 0. $
But $s$ is maximum.
Then if $|S| =|B(S)|-1 $, $S$ can be added to $SC$;
if  $|S| < |B(S)|-1$, $s$ would not be maximum and $S$ must be already in $SC$, contradiction.

\item $B(S)\subset C_i, 1\leq i \leq \omega(G-SC)$.  Consider the separator $SC\cup S $.
After the  removal of this separator, the component $C_i$ 
does not exist anymore and there are $B(S)+1$ new components.  
The scattering number would become:
 $s'=\omega(G- \{SC \cup S\})-|SC \cup S|= s+|B(S)|-|S|.$

By hypothesis, $|S| < | B (S) |$. 
So, $s$ would not be  maximum and $S$ must be already in $SC$. Contradiction.
\end{itemize}
 Hence, in any case, $SC \cup S$ is a scattering set of $G$.
   \qed\end{proof}\\

We notice that  a strictly chordal graph  that  satisfies the hypothesis of Theorem \ref{theo:notborder}  is a type B graph. 
Theorem \ref{theo:notborder2} and \ref{theo:notborder3}  concludes the study of border minimal vertex separators for type B graphs. 
 
  \begin{theorem}\label{theo:notborder2}

Let $G$ be a type B graph  with $|\mathbb{S}| >1$ and let $S\in {\mathbb S}$ be a border mvs
with  $|S| > |B(S)|$.  Then  $S$ is not a subset of  any  scattering set of $G$.
\end{theorem}
 
\begin{proof}
As $G$ is type B graph, $\tau(G)<1$. 
By Lemma \ref{lem:touscat}, $ sc(G)\geq 1$ and,  in addition,   $sc(G)\neq 1\therefore sc(G)\geq 2$.  

By hypothesis,   $|S| >  |B(S)|=\mu(S)$. So, $|S| =  \mu(S) +k$, $k\geq 1$. 
Consider $\mathbf S$  a separator of $G$  with  $\omega(G- \mathbf{S})-|\mathbf{S}|\geq 2$ such that $S \not\subset \mathbf S$. 
It is  immediate that $\omega(G- (\mathbf{S} \cup S))-|\mathbf{S} \cup S| <  \omega(G- \mathbf{S})-|\mathbf{S}|$.
 In particular, if $\mathbf S$ is a scattering set, $S$  is not a subset of $\mathbf S$.  
 \qed
\end{proof}
 
 \begin{theorem}\label{theo:notborder3}
Let $G$ be a type B graph  with $|\mathbb{S}| >1$, $S\in {\mathbb S}$ be a border mvs with $|S| = |B(S)|$. 
Then $sc(G)=sc(G')$ where $G'=G-(B(S)\setminus S)$. 
\end{theorem}

\begin{proof}
As $G$ is type B graph, by the definition, there is a \emph{mvs} $S'\in {\mathbb S}$  such that $|S'|< \mu(S')$. 
As $S$ is a border \emph{mvs} with  $|S|=|B(S)|$, it follows that  $|B(S)|=\mu(S)$ and  $S\neq S'$. 

The set $B(S)\setminus S$ consists of simplicial vertices of $G$. 
The graph $G'$ is obtained by removing these vertices. 
Therefore,  $S$ is a set of simplicial vertices of $G'$, $S'$ is a \emph{mvs} of $G'$ 
and $G'$ is also a type B graph with $sc(G')=s\geq 2$.   

Consider $SC'$ a scattering set of $G'$. It is  immediate that $S \not \subset SC'$. 
Let us analyse the separator $SC'\cup S$ of $G$. 
If there is a maximal clique $Q$ of $G$ such that $S\subset Q$ and $Q\subset SC' \cup S$, 
$\omega(G-\{ SC'\cup S \})-|SC' \cup S|=\omega(G-SC')-  |SC'|+ \omega(G- S)-|S|-2=s-1$.
Otherwise, $\omega(G-\{ SC'\cup S \})-|SC' \cup S|=\omega(G-SC')-  |SC'|+ \omega(G- S)-|S|-1=s$.
As $G'=G-(B(S)\setminus S)$, $SC'$ is also a scattering set of $G$. 
Then, $sc(G)=sc(G')$. 
\qed
\end{proof}

%%%%%%%%
%%%%%%   algoritmo
%%%%%%%%

%\subsection{determination of the scattering number of type B graphs}\label{subsection:typeB}

\smallskip

For the determination of the scattering number of a type B graph, 
firstly a scattering set of the graph must be determined;
the algorithm that performs this task consists of the application of Theorems   \ref{theo:notborder}, 
\ref{theo:notborder2} and \ref{theo:notborder3}.
So, after the computation of this set, it must be removed from the original graph 
and the remaining components accounted for, i.e., the scattering number is determined.

An intuitive algorithm can be drafted: we search graph $G$ looking for a
border \emph{mvs} $S$, that always exists (Theorem \ref{theo:specialS}).
If $|S| \geq |B(S)|$, $S$ can be ignored:  the vertices of $B(S)\setminus S$ must be removed from $G$ 
and the vertices of $S$ become simplicial vertices in the updated graph $G$.
If $|S| < |B(S)|$, the vertices of $S$  must be included  in the scattering set of $G$
and the vertices of $B(S)$ removed from the graph.
This step is  repeated until all minimal vertex separators are covered.
For each border \emph{mvs}, we must go through the updated graph again.
This algorithm has $O(nm)$ time complexity  since in a strictly chordal graph there are $O(n)$ 
minimal vertex separators and, for each one, the entire graph must be searched.
However, a  more efficient implementation can be presented.

\begin{figure}[h]
\hrule
\begin{myalgorithm}
\Algoritmo \em Scattering-set determination;  \\
{\bf Input:} $CB(G)$; \\
{\bf Output:}  $scattering$-$set(G)$;\\
\Inicio \\
\> Initialize arrays $card, status, entry$;\\
\> $SC, scattering$-$set(G) \gets \emptyset$;\\
\> $entry\_order \gets 0;$ \\
%\> \Para $v \in V$ \faca $\,entry(v) \gets 0$; \\
\> $root\gets v\in V$ such that $ status(v)=mvs$; $parent(root) = NULL$; \\ 
\> $dfs(root)$;\\
\> \Se $card(root) < \mu(root)$ \entao\\
\>\> $SC \gets SC \cup \{v\}$;  \\
\> \Para $v \in SC$ \faca  \hskip 1cm
 \% let $S$ be the $mvs$ represented by $v$ \\
\>\> $scattering$-$set(G) \gets scattering$-$set(G) \cup S$;\\
\> \\
\> \Procedimento $dfs(v)$; \\
\> \Inicio\\
\>\> $entry(v) \gets entry\_order \gets entry\_order +1$; \\
\>\> \Para $w \in Adj(v)$ \faca \\
\>\>\> \Se $entry(w) = 0$ \entao\\
\>\>\>\> $parent(w) \gets v$;\\
\>\>\>\> $dfs(w)$; \\
\>\> \Se $status(v) \neq mvs$ \entao               \hskip 6cm (*)\\                               
\>\>\> \Se $status(v) = false\_clique$ \entao\\
\>\>\>\> $\mu(parent(v)) \gets \mu(parent(v)) -1$;\\
\>\> \senao \Se $v \neq root$  \entao      \\
\>\>\> \Se $card(v) < \mu(v)$ \entao \\
\>\>\>\> $SC \gets SC \cup \{v\}$;  \\
\>\>\>\> \Se $card(parent(v)) = card(v) + card(parent(parent(v))$ \entao \hskip 0.4cm (**) \\
\>\>\>\>\> $status(parent(v)) \gets  false\_clique$ \\
\>\>\>\> \senao $card(parent(v)) \gets card(parent(v)) - card(v);$ \\
\> \Fim\\
\Fim.
\end{myalgorithm}
\hrule
\end{figure}

The algorithm {\em Scattering-set determination}, presented here,
relies on a depth-first search over the clique-bipartite graph $CB(G)=({\mathbb S} \cup {\mathbb Q}, F)$, 
which is a tree (Theorem \ref{theo:BC}).
 It is immediate to see that, although  the algorithm performs on the tree, 
it works at the same time with the structures that are represented by the vertices of the tree.
The search must begin in a vertex that represents a \emph{mvs} of $G$;  
all leaves of the depth-first search tree are vertices representing maximal cliques.

At each step of the algorithm, a vertex representing a maximal clique or a \emph{mvs} is considered.
They  take turns in the tree, since $CB(G)$ is a bipartite graph.
They are analysed at the moment that the vertex comes out from the recursion stack.
Observe that at this point all the action over its descendants is already performed, that is, 
the subgraph induced by them is already updated.

 The algorithm maintains labels for the vertices of $CB(G)$: 

$-\, entry(v)$: the order in which the vertex is visited;

$-\, parent(v)$: the parent of the vertex in the depth-first search tree;

$-\, card(v)$: the cardinality of the maximal clique or 
 the \emph{mvs} represented by  $v$;
 
$-\, \mu(v)$: equal to $\mu(S)$, being $S \in \mathbb S$ represented by $v$;
 
\smallskip 
  
$-\, status(v) = \left\{   \begin{array}{c l}
                          mvs,                                & v \mbox{   represents a } mvs \\
                          true\_clique,  &  v \mbox{ represents an existing clique }\\  
                                                & \,\,\,\,\mbox{  in the updated graph} \\
                          false\_clique ,   & v \mbox{ represents a clique that does not  }\\
                                                & \,\,\,\,\mbox{   exist anymore in the updated graph}\\
                        \end{array}
                    \right.  $

\medskip

 Let $v$ be the vertex to be analysed. 
\begin{itemize} 
\item  {\em If $v$  represents a maximal clique}:
  
If $v$ is a leaf of $CB(G)$, nothing happens. 
Otherwise, its status must be addressed; if it is  $false\_clique$ 
the number of adjacent cliques of $parent(v)$, which represents a $mvs$, $\mu(parent(v))$  must be decreased. 
Observe that this computation simulates that all the subgraph
induced by  $v$ and its descendants, i.e, the subgraph represented by these vertices on the graph, 
does not belong anymore to the graph.

\item {\em If $v$  represents a mvs $S$}:
 
 At the moment that vertex $v$ is analysed, it corresponds to a \emph{mvs} of the updated graph $G$.
 Its descendants in the depth search tree represent boundary cliques
 of $G$  or a subgraph of $G$ that does not contain any \emph{mvs} that must be in the scattering set.
 So, the \emph{mvs} can be considered as a border \emph{mvs} and
 Theorems  \ref{theo:notborder}, \ref{theo:notborder2} and \ref{theo:notborder3}
 can be applied.
We must consider two cases.
Firstly, if $|S| \geq  \mu(S)$:  the $mvs$ $S$ does not need to be included in the scattering set of $G$ and 
it is no more considered as a separator; nothing is done.
 Its vertices will be treated as simplicial vertices in the next iteration.
In the second case,  when  $|S| < \mu(S)$, the vertices of $S$ must be added to the scattering set of $G$
 and removed from the graph.
The graph must be updated:  the maximal clique  $Q$ that contains the $mvs$ and
that remains in the graph is analysed: 
if $Q$ is composed only by two minimal vertex separators, then $Q$, as a maximal clique, will not
exist anymore in the graph.
This computation is implemented in the algorithm by testing
 $card(parent(v))$ and, if needed, updating the label $status(parent(v))$ (line (**) of the algorithm).
 \end{itemize}

The algorithm determines $SC$, a set of vertices of $CB(G)$, each one representing a $mvs$  of $G$.
The scattering set of $G$ is the result of the union of these minimal vertex separators.

\subsection{time complexity of the determination of the scattering number}  

\begin{theorem}
The determination of the scattering number of a strictly chordal graph $G$  has linear time complexity.
\end{theorem}
\begin{proof}
Let $G=(V,E)$ be a strictly chordal graph. 
The set of maximal cliques, the set of minimal vertex separators of $G$ and their multiplicities must be determined;
 this can be accomplished in linear time complexity \cite{MP10}.
The determination of the toughness of the graph depends on a traversal of $\mathbb S$ that 
has $O(n+m)$ time complexity.

If $\tau(G) \geq 1$, by Theorem \ref{theo:tau>=1}, we must determine  a \emph{mvs} that satistisfies 
 $max_{_{S\in \mathbb S}}  \{ \mu(S) +1 -|S| \}$. 
As the minimal vertex separators and their multiplicities are already known, a new traversal of $\mathbb S$
is needed.
It has $O(n+m)$ time complexity.

If $\tau(G) < 1$ the partition of the set of graphs  must be establish.
In order to perform it, the  analysis of each \emph{mvs} is, again, necessary.
If  $G$ is a type A graph, the determination of the scattering number has constant time complexity,
by Theorem \ref{theo:sctype A}.
If  $G$ is a type B graph, the algorithm {\em Scattering-set determination} must be performed.
As seen, the algorithm relies on a depth-first search that has linear time complexity.
As $G$ is a strictly chordal graph, we know that 
for any distinct $S, S^{\prime} \in {\mathbb S}$, $S \cap S^{\prime}= \emptyset$ (Theorem \ref{theo:caract2}).
So, a vertex $v$ is a simplicial vertex or it belongs exactly to one $mvs$.
By labeling the vertices of  maximal cliques, 
it is possible to built $CB(G)$ in linear time complexity.
The vertex set of graph $CB(G)$ is ${\mathbb Q} \cup {\mathbb S}$.
The sets ${\mathbb Q}$  and ${\mathbb S}$ have each one at most $n$ elements; 
so, ${\mathbb Q} \cup {\mathbb S}$ has also $O(n)$ elements.
As $CB(G)$ is a tree it has at most $2n -1$ edges.
The if statement marked (*)  in the algorithm 
takes constant time for its execution.
So,  the building of $CB(G)$ and the depth-first search have $O(n+m)$ time complexity.
Hence the theorem is proved.
 \qed
\end{proof}

%*******************************************************************************
%   Bibliography                                *
%*******************************************************************************

\end{document}